\documentclass[11pt]{amsart}
\usepackage{amsmath,amsthm,amsfonts,latexsym}

\newtheorem{thm}{Theorem}[section]
\newtheorem{lem}[thm]{Lemma}
\newtheorem{pro}[thm]{Proposition}

\newtheorem{cor}[thm]{Corollary}
\newtheorem{cnj}[thm]{Conjecture}
\newtheorem{anz}[thm]{Ansatz}
\newtheorem{alg}[thm]{Algorithm}

\newcommand{\bpf}{\noindent {\em Proof.  }}
\newcommand{\epf}{\qed \vspace{+10pt}}
\newcommand{\Aut}{{\mathrm{Aut}}\;}
\newcommand{\al}{\alpha}
\newcommand{\be}{\beta}
\newcommand{\ga}{\gamma}
\newcommand{\si}{\sigma}

\newcommand{\ta}{\theta}
\newcommand{\sh}{{\mathrm{sh}}}
\newcommand{\ch}{{\mathrm{ch}}}

\newcommand{\xh}{\widehat{x}}

\newcommand{\mT}{\operatorname{\mathsf{T}}}

\newcommand{\sfC}{{\mathsf{K}}}
\newcommand{\vep}{\varepsilon}
\newcommand{\f}{\frac}
\newcommand{\tf}{\tfrac}
\newcommand{\ld}{\ldots}
\newcommand{\cd}{\cdots}
\newcommand{\de}{\delta}
\newcommand{\De}{\Delta}
\newcommand{\mcC}{\mathcal{C}}
\newcommand{\mcP}{\mathcal{P}}
\newcommand{\mcA}{\mathcal{A}}
 \newcommand{\sC}{\mathcal{K}}
\newcommand{\fS}{\mathfrak{S}}
\newcommand{\pa}{\partial}
\newcommand{\pe}{\prime}
\newcommand{\bfz}{{\bf z}}
\newcommand{\bfy}{{\bf y}}
\newcommand{\bfc}{{\bf c}}

\newcommand{\subp}[1]{{\mathsf{dompoly}}_{#1}}

\newcommand{\dmrjdel}[1]{}
 
\title{Transitive powers of Young-Jucys-Murphy elements are central}

\author[I.~P.~Goulden]{I.P. Goulden$^*$}
\author[D.~M.~Jackson]{D.M.~Jackson$^\dagger$}

\thanks{
${\hspace{-1ex}}^*$Department of Combinatorics and Optimization, 
                                                University of Waterloo, Waterloo, Ontario, Canada.;    \\
${\hspace{.35cm}}$ \texttt{ipgoulden@uwaterloo.ca}}

\thanks{
${\hspace{-1ex}}^\dagger$Department of Combinatorics and Optimization, 
                                                University of Waterloo, Waterloo, Ontario, Canada.;  \\
${\hspace{.35cm}}$ \texttt{dmjackso@math.uwaterloo.ca}}

\date{April 5, 2007}

\begin{document}
\maketitle

\begin{abstract}
Although powers of the Young-Jucys-Murphy
elements $X_i = (1\, i)+(2\, i)+\cd +(i-1\, i), \quad i=1,\ld ,n$, in the
symmetric group $\fS_n$ acting on $\{ 1,\ld ,n\}$ do not lie in the centre of the group algebra
of $\fS_n$, we show that transitive powers, namely the sum of the contributions from elements that
act transitively on~$[n],$ are central.  We determine the coefficients, which we call star factorization 
numbers, that occur in the resolution of transitive powers with respect to the class basis of the 
centre of $\fS_n,$ and show that they have a polynomiality property.
These centrality and polynomiality properties have  seemingly unrelated consequences.
First, they answer a question raised by Pak~\cite{p} about reduced decompositions; second, they
explain and extend the beautiful symmetry result discovered by Irving and Rattan~\cite{ir};
and thirdly, we relate the polynomiality to an existing
polynomiality result for a class of double Hurwitz numbers
associated with branched covers of the sphere, which therefore suggests that there may be an
ELSV-type formula (see \cite{elsv1}) associated with the star factorization numbers.

\dmrjdel{
A star transposition on $[n]=\{ 1,\ld ,n\}$ is a transposition of the form $(1\, a)$, for $a=2,\ld ,n$. 
The factorization $\si=\tau_1\cd\tau_r$ of a permutation $\si$ into star transpositions 
$(\tau_1,\ld ,\tau_r)$ is {\em transitive} when the group generated by the factors $\tau_1,\ld ,\tau_r$ 
acts transitively on $[n]$. A transitive factorization with $2g$ more than the minimum number of factors 
is said to have \emph{genus} $g$, for $g\ge 0$.
We enumerate transitive factorizations of an arbitrary permutation~$\si$ into star transpositions for 
all genera. In each case, the counting
formula is symmetric in the parts of $\al$ where $\si$ is in the conjugacy
class $\sC_\al.$ This extends a result of Pak, who considered one class of permutations in genus $0$, 
and recent work of Irving and Rattan, who gave the complete solution for genus $0$. In both cases, this 
previous work uses a combinatorial bijection. Our method uses a linear partial differential equation for 
the generating series, derived by a join-cut analysis for the left-most factor.
}
\end{abstract}


\section{Introduction and background}\label{intro}
We begin with an account of the main theorem of this paper and its
relationship to the enumeration of a class of ramified covers of the
sphere, a question that arises in algebraic geometry. 

\subsection{Young-Jucys-Murphy elements and the Main Theorem}

The \emph{Young-Jucys-Murphy elements} in the group algebra ${\mathbb{C}}\fS_n$ 
of  the symmetric group $\fS_n$ on $[n]:=\{ 1,\ld ,n\}$,
are given by
$$X_i = (1\, i)+(2\, i)+\cd +(i-1\, i), \qquad i=1,\ld ,n,$$ 
where $X_1=0$ (see, \textit{e.g.},~\cite{vo} for a detailed description and further
references). Let $Z(n)$ denote the centre of ${\mathbb{C}}\fS_n$,~$n\ge 1$. Then the
algebra generated by $Z(1),\ld ,Z(n)$ is called the \emph{Gel'fand-Tsetlin algebra},
and one of the key results described in~\cite{vo} is the fact that this algebra
is also generated by $X_1,\ld ,X_n$, despite the fact that $X_n$ is clearly
\emph{not} contained in $Z(n)$ for any $n>2$.

We define a linear operator $\mT$ on ${\mathbb{C}}\fS_n$ by
$\mT(\si_1\cd\si_r)= \si_1\cd\si_r$, if the group generated by the permutations $\si_1, \cd, \si_r$ acts
transitively on $[n]$, and  $\mT(\si_1\cd\si_r) = 0$ otherwise.
The subject of this paper is $\mT\, X_n^r$, for
an arbitrary non-zero integer $r$,  which we call
a \emph{transitive power} of $X_n$.
It is straightforward matter to apply $\mT$ to $X_n^r=\left(\sum_{i\in [n-1]} (i\, n)\right)^r$,
since the only products not annihilated are those containing at least one
occurrence of $(i\, n)$ as a factor for every $i\in [n-1]$. It
follows immediately from the Principle of  Inclusion-Exclusion that a transitive power can be written
explicitly as
$$\mT\, X_n^r = \sum_{\ga\subseteq [n-1]}
(-1)^{|\ga|}X_n({\overline{\ga}})^r,$$
where $X_n(\ga)=\sum_{j\in\ga}(j\, n)$.

Our main result, Theorem~\ref{t2}, is that the transitive powers of $X_n,$ unlike powers,  
\emph{are} contained in $Z(n)$. Moreover, since  a basis for $Z(n)$ is
given by 
the set of all $\sfC_{\al}$  where $\sfC_{\al} =\sum_{\pi\in\sC_\al} \pi$, and $\sC_\al$ is
the conjugacy class of $\fS_n$ (naturally) indexed by the partition $\al$ of $n$,
then Theorem~\ref{t2} expresses $\mT\, X_n^r$ as
an explicit linear combination of the $\sfC_{\al}$.

We use the following notation and terminology for partitions.
If~$\al_1,\cd,\al_k$ are positive integers with~$1\le \al_1\le\cd\le\al_k$ and~$\al_1+\cd +\al_k=n$, 
then~$\al=(\al_1,\ld ,\al_k)$ is a \emph{partition} of $n$ with $k$ \emph{parts},
and we write $\al\vdash n$ and $l(\al )=k$, for $n,k\ge 0$. Let $\al\setminus\al_j$ denote
the partition obtained by removing the single part $\al_j$ from $\al$, for any $j=1,\ld ,k$.
Let $\al\cup m$  denote the partition obtained by inserting a single new part equal
to $m$ into $\al$ (placed in the appropriate ordered position).
Let $2\al=(2\al_2,\ld,2\al_k)$, and $a_{\al}=a_{\al_1}\cd a_{\al_k}$ for any
indeterminates $a_1,a_2,\ld$.
Let $\mcP$ denote the set of all partitions,
including the empty partition $\vep$, which it a partition of $0$ with $0$ parts.
If $\al$ has $f_j$ parts
equal to $j$ for each $j\ge 1$, then we also
use $(1^{f_1}2^{f_2}\cd )$ to denote $\al$, and we write  $|\Aut\al |=\prod_{j\ge 1} f_j!$ 

In the statement of our main result we use the notation $p_i \equiv p_i(\al)$ to denote
the $i$-th power sum of the parts of the partition $\al$, $i\ge 1$,
and $q_i\equiv q_i(\al):=p_i+p_1-2$, $i\ge2,$ and we define $\xi_{2j}$ and $\xi$ by
\begin{equation}\label{xidef}
\sum_{j\ge 1}\xi_{2j}x^{2j}:=\log\left(\xi(x)\right),\qquad\mbox{where}\qquad
\xi(x):= 2 x^{-1}\sh\left(\tf{1}{2}x\right), 
\end{equation}
where \emph{sh} and \emph{ch} denote, respectively, hyperbolic sine and cosine. 

\begin{thm}[Main Theorem]\label{t2}\label{t3}
For $r\ge 0$,  $\mT X_n^r$ is contained in the centre $Z(n)$ of $\mathbb{C}\fS_n$.
Moreover, the resolution of  $\mT X_n^r$ with respect to the class basis  of $Z(n)$ is
$$\mT X_n^r = \sum_{\substack{\al\vdash n,\\g\ge 0}} a_g(\al )\sfC_{\al}$$
where the range of summation on the right hand side is restricted by the condition $n+m-2+2g=r$,
with $m=l(\al)$, and $a_g(\al)$ is a polynomial in the parts of $\al$ given by
$$a_g(\al) =\f{1}{n!} {(n+m-2+2g)!} \, \al_1\cd\al_m\,Q_g(\al),\qquad\mbox{where}\qquad
Q_g:=\sum_{\be\vdash g}\f{\xi_{2\be}q_{2\be}}{|\Aut\be|}, \qquad g\ge 0.$$
\end{thm}
\vspace{.05in}

For example, the explicit expressions for small genera $g=0, \ldots, 5$ are
\begin{eqnarray*}
Q_0 &=& 1,\quad 
Q_1 = \f{1}{24}q_2,\quad 
Q_2 = \f{1}{5760}\left(-2q_4+5q_{2^2}\right),  \quad 
Q_3 = \f{1}{2^3 9!} \left(16q_6-42q_{4\, 2}+35q_{2^3}\right),  \\
Q_4 &=& \f{1}{3\cdot2^710!} \left(-144q_8+320q_{6\,2}+84q_{4^2}-420q_{4\,2^2}+175q_{2^4}\right), \\
Q_5 &=& \f{1}{3\cdot2^812!}
                 \left(768q_{10}-1584q_{8\,2}-704q_{6\,4}+1760q_{6\,2^2}+924q_{4^2\,2}-1540q_{4\,2^3}
                       +385q_{2^5}\right),
\end{eqnarray*}

\dmrjdel{
Q_6 &=& \f{1}{3^2 2^7 16!}
(-1061376 q_{12}+2096640 q_{10\, 2}+864864 q_{8\, 4}-2162160 q_{8\,2^2}
+366080 q_{6^2}\\
&&-1921920 q_{6\, 4\, 2}+1601600 q_{6\, 2^3}-168168 q_{4^3}+1261260 q_{4^2\, 2^2}
-1051050 q_{4\, 2^4}+175175 q_{2^6}),\\
Q_7&=&   \f{1}{3^3 2^{10} 16!}
(552960 q_{14}
-1061376 q_{12 \,2} 
-419328 q_{10 \,4}
+1048320 q_{10 \,2^2}
-329472 q_{8 \,6}  \\
&\mbox{}&
+864864 q_{8 \,4 \,2}
-720720 q_{8\,2^3}
+366080 q_{6^2 \,2}
+192192 q_{6\,4^2}
-960960 q_{6\,4\,2^2}
+400400 q_{6\,2^4} \\
&\mbox{}&
-168168 q_{4^3\,2} 
+420420 q_{4^2\,2^3}            
-210210 q_{4\,2^5} 
+25025 q_{2^7}), \\
Q_8 &=& \f{1}{5^{17}3^4 2^{16} 16!}
(-200005632 q_{16}
+376012800 q_{14 \,2}
+144347136 q_{12 \,4}
-360867840 q_{12 \,2^2}  \\
&\mbox{}& +108625920 q_{10 \,6}
-285143040 q_{10 \,4 \,2}
+237619200 q_{10 \,2^3}
+50409216 q_{8^2}
-224040960 q_{8 \,6 \,2}  \\
&\mbox{}& 
-58810752 q_{8 \,4^2}
+294053760 q_{8 \, 4 \,2^2}
-122522400 q_{8} q_{2}^4
-49786880 q_{6^2 \,4}
+124467200 q_{6^2\, 2^2}  \\
&\mbox{}& 
+130690560 q_{6 \,4^2 \,2}
-217817600 q_{6 \,4 \,2^3}
+54454400 q_{6 \,2^5}
+5717712 q_{4^4}
-57177120 q_{4^3 \,2^2}  \\
&\mbox{}& 
+71471400 q_{4^2\, 2^4}
-23823800 q_{4 \,2^6}
+2127125 q_{2^8}).
}


\subsection{Background}
\subsubsection{Minimal factorizations into star transpositions}

We now turn our attention temporarily to a another point of view.
The transpositions $(1\, a)$, for $a=2,\ld ,n$, are called {\em star} transpositions in~$\fS_n$,
with the distinguished element~$1$ (it appears in each transposition) referred to as the 
\emph{pivot} element.
An ordered factorization~$(\tau_1,\ld ,\tau_r)$ of~$\si\in\fS_n$ into  star transpositions is said to
be \emph{transitive} if the group generated by $\tau_1,\ld ,\tau_r$ acts transitively on~$[n].$
For a transitive factorization of $\si\in\sC_{\al}$ into $r$ star transpositions, a result in~\cite{gj0}
implies that~$r=n+m-2+2g$ for some non-negative integer~$g,$ where $\al$ has $m$ parts.
Thus~$r\ge n+m-2$, and we refer to transitive factorizations into~$n+m-2$ star transpositions
as \emph{minimal}.

Pak~\cite{p} enumerated minimal factorizations (he called them \emph{reduced decompositions})
into star transpositions for permutations fixing the pivot element~$1$, with
exactly~$m$ other cycles, each of length $k\ge 2$. 
More recently, Irving and Rattan~\cite{ir} 
generalized Pak's result by considering minimal factorizations of arbitrary permutations 
into star transpositions, and proved the following elegant result.

\begin{thm}[\cite{ir}] \label{t1}
For each permutation $\si\in\sC_{\al}$ with~$\al=(\al_1,\ld ,\al_m)$, $\al_1+\cd +\al_m=n$ and~$m,n\ge 1$, the
number of transitive factorizations of $\si$ into $n+m-2$ star transpositions  is
$$\f{(n+m-2)!}{n!}\,\al_1\cd\al_m.$$
\end{thm}

Because of the apparent asymmetry of these factorizations (\textit{i.e.}, the pivot element $1$ appears
in \emph{every factor}), the fact that Theorem~\ref{t1} is constant on
conjugacy classes is particularly surprising (we shall refer to this fact
as the \emph{centrality property} of Theorem~\ref{t1}). 
The proofs given in~\cite{p} and~\cite{ir} are bijective, involving restricted words and plane trees.

In terms of factorizations into star transpositions, the number $a_g(\al)$ given by the Main Theorem
clearly can be interpreted as the number of transitive factorizations of
each $\si\in\sC_{\al}$ into $n+m-2+2g$ star transpositions, with pivot element $n$.
We shall call $a_g(\al)$ a \emph{star factorization number}.
Thus Theorem~\ref{t1} is precisely the case $g=0$ of Theorem~\ref{t2} (the necessary
relabelling of the pivot element is justified by the centrality of these results).
The investigation described in this paper answers Pak's~\cite{p} question about an explicit expression
for the general case.  It was  motivated by Irving and Rattan's paper, in our
attempt to determine whether the centrality of their remarkable result for
star factorizations with a minimum number of factors persisted for star factorizations with
an \emph{arbitrary} number of factors.


\subsubsection{Connections with algebraic geometry}
The connection to algebraic geometry is made through Hurwitz's encoding~\cite{h}
of an $n$-sheeted branched cover of the sphere in terms of transpositions that represent
the sheet transitions at the elementary branch points.  In this context, the transitivity
of the factorizations corresponds to the connectedness of the cover.
From this perspective, the coefficient $a_g(\al)$ in Theorem~\ref{t2}  
counts genus $g$ branched covers of the sphere in which
the branching over the point $0$ is specified by $\al$, and there
are $n+m-2+2g$ other simple branch points, each of which corresponds to a transition between
sheet number $n$ (the \emph{pivot} sheet) and another sheet.
For the corresponding transitive factorizations into star transpositions, we therefore also refer
to $g$ as the \emph{genus of the factorization} (\emph{e.g.}, Theorem~\ref{t1}
counts genus $0$ factorizations). For further details about branched covers,
see, for example, ~\cite{gj0}, \cite{gjvai}, \cite{gjv} and \cite{h}.

The \emph{double Hurwitz number} $H^g_{(n),\al}$ is equal
to the number of genus $g$ branched covers
of the sphere in which the branching over
the points $0$ and $\infty$ is specified by $(n)$ and $\al$, respectively,
together with $m-1+2g$ other simple branch points. 
A scaling of this double Hurwitz number to 
$$b_g(\al):=\al_1\cd\al_m H^g_{(n),\al}$$
gives the number of transitive factorizations of each $\si\in\sfC_{\al}$ into $m-1+2g$ transpositions
and a single $n$-cycle. There is a striking similarity between Theorem~\ref{t2} and the following
result, in which the notation ${\widehat{q}}_i:=p_i-1$, $i\ge 1$ is used.

\begin{thm} [\cite{gjv}] \label{t2sim}
For $r\ge 0,$ the resolution of $\sfC_{(1^{n-2}\,2)}^r\sfC_{n}$  with respect to the class basis 
of the centre $Z(n)$ of $\mathbb{C}\fS_n$ is
$$\sfC_{(1^{n-2}\,2)}^r\sfC_{n}=\sum_{\substack{\al\vdash n,\\g\ge 0}} b_g(\al )\sfC_{\al},$$
where the range of summation on the right hand side is restricted by the condition $m-1+2g=r$,
with $m=l(\al)$, and $b_g(\al)$ is a polynomial in the parts of $\al$ given by
$$b_g(\al)=(m-1+2g)!\, n^{m-2+2g}\,\al_1\cd\al_m\,{\widehat{Q}}_g(\al),\qquad
\mbox{where}\qquad
{\widehat{Q}}_g:=\sum_{\be\vdash g}\f{\xi_{2\be}{\widehat{q}}_{2\be}}{|\Aut\be|}, \qquad g\ge 0.$$
\end{thm}
\vspace{.05in}

This is a restatement of Theorem~3.1 in~\cite{gjv}, which gives a
formula for the double Hurwitz number $H^g_{(n),\al},$ since
$\sfC_{1^{n-2}2}^r\sfC_{n} =\mT(\sfC_{1^{n-2}2}^r\sfC_{n})$ (each term in 
$\sfC_{n}$ acts transitively on $[n]$).

\subsubsection{Two relationships between  Theorems~\ref{t2} and~\ref{t2sim}}
To explore a more direct relationship between Theorems~\ref{t2} and~\ref{t2sim},
we now give two expressions for $a_g(\al)$ in terms of the $b_h(\ga)$'s.

The first  is  very simple and expresses $a_g(\al),$ which
enumerates factorizations in $\fS_n,$ directly in terms of $b_g(\al\cup 1^{n-1}),$ which
enumerates factorizations in $\fS_{2n-1}.$

\begin{cor}\label{connS2n-1}
For $g\ge 0$ and $\al$ a partition of $n$ with $m$ parts, we have
$$a_g(\al)=\f{1}{n!(2n-1)^{n+m-3+2g}}b_g(\al\cup 1^{n-1}).$$
\end{cor}

\bpf
In the notation of Theorems~\ref{t2} and~\ref{t2sim},
clearly $q_i(\al)={\widehat{q}}_i(\al\cup 1^{n-1})$,
so $Q_g(\al)={\widehat{Q}}_g(\al\cup 1^{n-1})$. The
result follows immediately from Theorems~\ref{t2} and~\ref{t2sim}.
\epf

The second expresses $a_g(\al)$ as a linear combination of $b_{g-h}(\al)$, $0\le h\le g$, 
each of which enumerates factorizations in $\fS_n$. 

\begin{cor}\label{connSn}
For $g\ge 0$ and $\al,$ a partition of $n$ with $m$ parts, we have
$$a_g(\al)=\f{1}{n!}
\sum_{h=0}^g\f{b_{g-h}(\al)}{n^{m-2+2g-2h}}
\binom{n+m-2+2g}{n-1+2h}
\sum_{j=0}^{n-1}\binom{n-1}{j}(-1)^j
\left(\tf{1}{2}(n-1)-j\right)^{n-1+2h}.$$
\end{cor}

\bpf
In the notation of Theorems~\ref{t2} and~\ref{t2sim},
clearly $q_i(\al)={\widehat{q}}_i(\al)+n-1$. Then
from Theorems~\ref{t2} and~\ref{t2sim}, and~(\ref{xidef}), we have
\begin{equation}
\sum_{g\ge 0}Q_g(\al) x^{2g}=
\exp\left(\sum_{j\ge 1}\xi_{2j}q_{2j}(\al)x^{2j}\right)
=\xi(x)^{n-1}\sum_{g\ge 0}{\widehat{Q}}_g(\al) x^{2g}.\label{QQhat}
\end{equation}
But, for $h\ge 0$, we have (using the notation $[A]B$ to denote the
\emph{coefficient} of $A$ in $B$)
\begin{equation*}
[x^{2h}]\xi(x)^{n-1}=[x^{n-1+2h}]\left( e^{\f{x}{2}}-e^{-\f{x}{2}}\right)^{n-1}
=\sum_{j=0}^{n-1}\binom{n-1}{j}(-1)^j\f{\left(\tf{1}{2}(n-1)-j\right)^{n-1+2h}}{(n-1+2h)!}
\end{equation*}
and, together with~(\ref{QQhat}), this gives
$$Q_g(\al)=\sum_{h=0}^g\f{{\widehat{Q}}_{g-h}(\al)}{(n-1+2h)!}
\sum_{j=0}^{n-1}\binom{n-1}{j}(-1)^j
\left(\tf{1}{2}(n-1)-j\right)^{n-1+2h}.$$
The result follows immediately from Theorems~\ref{t2} and~\ref{t2sim}.
\epf

\dmrjdel{
We note that, in Theorem~\ref{t2}, the divisors~$24,$ $5760$  and~$2^3\,9!$
that appear, respectively,  in the  genus~$1,$ ~$2$ and~$3$ counting formulae, 
also appear  in the {\em Hurwitz numbers} counting genus~$1,$ ~$2$ and~$3$ ramified covers 
of the sphere. The reader is referred to~[\cite{gj1},\cite{gj2},\cite{gjv}], for
example,  for the appearance of these
divisors  in these Hurwitz numbers, for which the factors are \emph{arbitrary} transpositions. 
}

\subsection{Outline}

In Section 2, we introduce a generating series for the number of transitive
factorizations into star transpositions in arbitrary genus, and prove that
it is the unique formal power series solution of a linear partial differential
equation that we call the {\em Join-cut Equation} for this class of
factorizations. 
The proof is based on
a join-cut analysis of these factorizations, since the left-most factor $\si$
either joins two cycles of the product $\pi$ of the remaining factors to form one cycle
or cuts one
cycle of $\pi$ into two, depending on whether the two elements moved by $\si$
are, respectively, in different cycles of $\pi,$ or in the same cycle.
This approach has been applied previously where the factors are arbitrary
transpositions, for the genus $0$ case in~\cite{gj0}, and for arbitrary
genus in~\cite{gjvai} and \cite{gjv}.

In Section 3, we solve the Join-cut Equation to obtain the generating series
for transitive factorizations into star transpositions in arbitrary genus.
Then, by determining the coefficients in this generating series, we prove
Theorem~\ref{t2} (and hence also give a new proof of Theorem~\ref{t1}).

In Section 4, we pose some questions that arise from
this investigation, but that we have been unable to resolve.

\section{The Join-cut Equation}

Let $\fS_{\mcA}$ denote the symmetric group on an arbitrary set $\mcA$. 
For an arbitrary set $\mcA$ of size $n$ containing $1$ (for convenience,
we shall consider star transpositions with pivot element $1$),
let $\sC^{(i)}_{\al}$ denote the set of all permutations in $\fS_{\mcA}$ in
which $1$ lies on a cycle of length $i$ and the remaining
cycle-lengths in the disjoint cycle representation are given by
the parts of $\al$, where $\al\vdash n-i$, for $n\ge i\ge 1$.
It is straightforward to determine, independently of the choice of $\mcA$,  that
\begin{equation}\label{cardiclass}
|\sC^{(i)}_{\al}|={n-1\choose i-1}(i-1)!|\sC_{\al}|=
\f{(n-1)!}{\al_1\cd \al_k |\Aut\al |},
\end{equation}
where $\al=(\al_1,\ld ,\al_k)$.
Consider a fixed  permutation $\si\in\sC^{(i)}_{\al}$ in $\fS_n$, and 
let $c_g(i,\al )$ be the number of transitive factorizations
of $\si$ into $n+k-1+2g$ star transpositions (this number
is constant for each such $\si$ because of the symmetry of
elements $2,\ld ,n$; note that $\si$ lies in the conjugacy
class~$\sC_{\al\cup i}$, which has $m=k+1$ cycles).
Let $\Psi$ denote the generating series 
\begin{equation}\label{gendef}
\Psi(t,u,x;\bfz,\bfy):=\sum_{\substack{n\ge i\ge 1, \\k,g\ge 0}}
n t^n \f{u^{n+k-1+2g}}{(n+k-1+2g)!} x^{2g} z_i
\sum_{\substack{\al\in\mcP\\ \al\vdash n-i,\\l(\al)=k}}
|\sC^{(i)}_{\al}|c_g(i,\al ) y_{\al}.
\end{equation}
The following result is the Join-cut Equation for the set of transitive
factorizations into star transpositions. It states that
$\Psi$ is annihilated by the partial differential operator
\begin{equation}\label{E_opsDd}
\De:=\f{\pa}{\pa u}-t\f{\pa}{\pa t}t\sum_{i\ge 1}z_{i+1}\f{\pa}{\pa z_i}
-\sum_{i,j\ge 1}z_i y_j \f{\pa}{\pa z_{i+j}}
-x^2\sum_{i,j\geq 1}jz_{i+j}\frac{\partial ^{2}}{\partial z_{i}\partial
y_{j}}.
\end{equation}

\begin{thm}[Join-cut Equation]\label{jcut}
The generating series $\Psi=\Psi(t,u,x;\bfz,\bfy)$ is the unique formal power series solution
of $\De \Psi =0$,
with initial condition $\Psi(t,0,x;\bfz,\bfy)=z_1 t$.
\end{thm}

\bpf
Fix a triple $(k,g,i)$ of integers with $k,g\ge 0$ and $i\ge1$ to be other than~$(0,0,1).$ Also fix 
a partition $\al$ with $l(\al )=k$ and a permutation $\si\in\sC^{(i)}_{\al}$ in $\fS_n$,
where $\al\vdash n-i$. Consider a transitive factorization $(\tau_1,\ld ,\tau_r)$ of $\si$ into
star transpositions, where $r=n+k-1+2g$.
For this factorization, we let $\pi=\tau_2\cd \tau_r=\tau_1\si$, and $\tau_1=(1\, a)$.
There are $c_g(i,\al )$ such factorizations of $\si$, and we obtain a recurrence equation
for $c_g(i,\al )$ by considering the following case analysis for these
factorizations which is based on the left-most factor $\tau_1$.
\vspace{.05in}

\noindent
\underline{{\bf Case 1:}} $\tau_1\ne\tau_j$ for any $j=2,\ld ,r$. In this case, the
element $a$ is a fixed point in $\pi$, and $(\tau_2,\ld ,\tau_r)$ is
{\em not} a transitive factorization of $\pi$. But, if we
let $\pi'\in\fS_{[n]\setminus\{ a\}}$, whose
disjoint cycle representation is obtained by removing the one-cycle containing $a$ from
the disjoint cycles of $\pi$, then $(\tau_2,\ld ,\tau_r)$ is
a transitive factorization of $\pi'$.
But $\si$ is obtained from $\pi'$ by inserting $a$ immediately
before $1$ in the cycle of $\pi'$ containing $1$. This implies
that $\pi'\in\sC^{(i-1)}_{\al}$ in $\fS_{[n]\setminus\{ a\}}$.
Note that the transitive
factorization $(\tau_2,\ld ,\tau_r)$ of $\pi'$ has $r-1=(n-1)+k-1+2g$ factors and that
this is reversible, so
we conclude that the number of such factorizations is $c_g(i-1,\al )$,
the contribution from this case.
\vspace{.05in}

\noindent
\underline{{\bf Case 2:}} $\tau_1 =\tau_j$ for some $j=2,\ld ,r$. In this case, $(\tau_2,\ld ,\tau_r)$ is
a transitive factorization of $\pi$, since for a product of star transpositions in $\fS_n$ to
be transitive, it is necessary and sufficient that each of $(1\, 2),\ld ,(1\, n)$ appears
at least once as a factor (as observed in~\cite{ir}). There are two subcases, based
on which disjoint cycles of $\pi$ contain elements $1$ and $a$.
\vspace{.05in}

\noindent
\qquad{\bf Subcase 2(a):} $1$ and $a$ appear on the same cycle of $\pi$. In this subcase,
that cycle of $\pi$ is {\em cut} into two cycles in $\si$, one containing $1$,
and the other containing $a$. Consequently, for each
factorization $(\tau_1,\ld ,\tau_r)$ of $\si$, we obtain a factorization of $\pi$
 in this subcase by selecting $a$ to be any element on the $k$ cycles of $\si$ not
containing $1$. We account for the choices of $a$ on these cycles as follows: Suppose
the cycles are indexed so they have lengths $\al_1,\ld ,\al_k$
(the cycles are all non-empty, so they are distinguishable, even if their lengths are equal).
If $a$ is on the $j$th such cycle, of length $\al_j$, then there are $\al_j$ choices
of $a$, and the cycle of $\pi$ containing $1$ has length $i+\al_j$,
so we have $\pi\in\sC^{(i+\al_j)}_{\al\setminus\al_j}$ in $\fS_n$. Since
the transitive factorization $(\tau_2,\ld ,\tau_r)$ of $\pi$ has $r-1=n+(k-1)-1+2g$ factors
and this is reversible,
we conclude that there are $c_g(i+\al_j,\al\setminus\al_j )$ such factorizations, giving a
total contribution from this subcase of $\sum_{j=1}^k \al_j c_g(i+\al_j,\al\setminus \al_j )$.
\vspace{.05in}

\noindent
\qquad{\bf Subcase 2(b):} $1$ and $a$ appear on different cycles of $\pi$. In this subcase,
these cycles of $\pi$ are {\em joined} into a single cycle of $\si$, containing
both $1$ and $a$.  Consequently, for each
factorization $(\tau_1,\ld ,\tau_r)$ of $\si$, we obtain a factorization of $\pi$
in this subcase by selecting $a$ to be any other element on the cycle
of $\si$ containing $1$. We account for these $i-1$ choices of $a$ as
follows: Suppose that the cycle of $\si$ containing $1$, in cyclic order,
is $(1\, j_{i-1}\, \ld j_{1})$  (\textit{i.e.,} so $\si(1)=j_{i-1}$, $\si(j_{t})=j_{t-1}$,
for $t=2,\ld ,i-1$, and $\si(j_{1})=1$).  If $a=j_m$, then $\pi$ has disjoint
cycles $(1\, j_{i-1}\, \ld j_{m+1})$ (containing $1$) and $(j_{m}\, \ld j_{1})$, together with all
the cycles of $\si$ not containing $1$, so we 
have $\pi\in\sC^{(i-m)}_{\al\cup m}$ in $\fS_n$, and
the transitive factorization $(\tau_2,\ld ,\tau_r)$ of $\pi$ has $r-1=n+(k+1)-1+2(g-1)$ factors.
Since this is reversible,
we conclude that there are $c_{g-1}(i-m,\al\cup m)$ such factorizations, giving
a total contribution from this subcase of $\sum_{m = 1}^{i-1} c_{g-1}(i-m,\al\cup m)$.
\vspace{.05in}

Adding together the contributions from these disjoint cases, we obtain 
the linear recurrence equation
$$
c_g(i,\al )=c_g(i-1,\al )+\sum_{j=1}^k \al_j c_g(i+\al_j,\al\setminus \al_j )
+\sum_{m = 1}^{i-1} c_{g-1}(i-m,\al\cup m),
$$
for $k,g\ge 0$, $i\ge 1$ (except
for the simultaneous choices $k=g=0$ {\em and} $i=1$) and $\al$ with $l(\al)=k$.
The partial differential equation follows by multiplying this recurrence equation
by $n t^n \f{u^{n+k-2+2g}}{(n+k-2+2g)!} x^g z_i
|\sC^{(i)}_{\al}|y_{\al}$, and summing over the above range of $k,g,i,\al$.

The initial condition follows from the fact that there is a single, empty
factorization with no factors, of the single permutation (with $1$ as a fixed
point) in $\fS_1$. Thus we have $c_0(1,\vep)=1$.
\epf


\section{A proof of Theorem~\ref{t2}}

\subsection{An explicit solution to the Join-cut Equation}

The next result gives the explicit solution of the Join-cut Equation in terms of the
series $\xi$ defined in~(\ref{xidef}) and $W\equiv W(t,u,x: \bfz)$ where
$$W:= \sum_{\ell \ge 1} z_{\ell}\xi(\ell ux) \xi(ux)^{\ell -2} u^{\ell-1}t^{\ell}.$$

\begin{thm}\label{gsallg}
Let
$Z:=t\f{\partial}{\partial t} W(t,u,x: \bfz) \quad\mbox{and}\quad 
Y:= \xi (ux)^2u^2W(t,u,x: \bfy).$
Then
$\Psi = Ze^Y.$
\dmrjdel{
$$\Psi= \Big(\sum_{\ell \ge 1} \ell z_{\ell}\xi(\ell ux) \xi(ux)^{\ell -2}
u^{\ell-1}t^{\ell}\Big)\cdot
\exp\!\Big(\sum_{m \ge 1}  y_m \xi(mux) \xi(ux)^m u^{m+1}t^m\Big).$$}
\end{thm}

\bpf
It is a straightforward matter to show that the Join-cut Equation with the given
boundary condition has a unique solution.  The remainder of the proof is a verification that $\Delta$ 
annihilates $\Psi$ and that the boundary condition is satisfied.
\dmrjdel{
In the proof, we use the following
easily established facts.
\begin{equation}\label{xipe}
\xi^{\pe}(x)=x^{-1}\ch(\tf{1}{2}x)-2x^{-2}\sh(\tf{1}{2}x),
\end{equation}
\begin{equation}\label{shadd}
\sh(a\pm b)=\sh(a)\ch(b)\pm\ch(a)\sh(b).
\end{equation}
}  
\dmrjdel{
Let $Z:=\sum_{\ell \ge 1} \ell z_{\ell}\xi(\ell ux) \xi(ux)^{\ell -2} u^{\ell-1}t^{\ell}$,
 $\;\;Y:=\sum_{m \ge 1}  y_m \xi(mux) \xi(ux)^m u^{m+1}t^m$, and $\Phi:=Ze^Y$. 
 }  
 
The operator $\Delta$ is a linear combination of four differential operators.
It is straightforward to obtain the four expressions for the application of each of these
operators to $\Psi.$ Let $\xh:=ux$ for brevity. Then the expressions are:  
\begin{eqnarray*}
e^{-Y}\f{\pa \Phi}{\pa u}&=&\sum_{\ell \ge 1} \ell z_{\ell}\xi(\xh)^{\ell-3} u^{\ell-2}t^{\ell}
\Big((\ell-1) \xi(\ell \xh) \xi(\xh)+\ell \xh \xi^{\pe}(\ell \xh) \xi(\xh)+(\ell -2)\xh
\xi(\ell \xh) \xi^{\pe}(\xh)
\Big) \\
&\mbox{}&+Z\sum_{m\ge 1} y_m \xi(\xh)^{m-1} u^m t^m
\Big( (m+1)\xi(m\xh)\xi(\xh)+m\xh \xi^{\pe}(m\xh) \xi(\xh) +m\xh\xi(m\xh) \xi^{\pe}(\xh)
\Big),
\end{eqnarray*}
\begin{eqnarray}
e^{-Y}t\f{\pa}{\pa t}t\sum_{i\ge 1}z_{i+1}\f{\pa\Phi}{\pa z_i}
&=&\sum_{i\ge 1} i z_{i+1}\xi(i\xh)\xi(\xh)^{i-2} u^{i-1}t^{i+1}
\Big( i+1+\sum_{m\ge 1}my_m\xi(m\xh)\xi(\xh)^m u^{m+1} t^m
\Big),  \nonumber\\
e^{-Y}\sum_{i,j\ge 1} z_i y_j\f{\pa\Phi}{\pa z_{i+j}}&=&
\sum_{i,j\ge 1}(i+j) z_i y_j \xi\big( (i+j)\xh\big) \xi(\xh)^{i+j-2} u^{i+j-1} t^{i+j}, \nonumber \\
e^{-Y} x^2 \sum_{i,j\ge 1} j z_{i+j} \f{\pa^2\Phi}{\pa z_i y_j}&=&
x^2\sum_{i,j\ge 1} i j z_{i+j} \xi(i\xh)\xi(j\xh)\xi(\xh)^{i+j-2} u^{i+j} t^{i+j} \nonumber \\
&=& \sum_{\ell\ge 1} z_{\ell} \xi(\xh)^{\ell-2} u^{\ell-2} t^{\ell} S_{\ell},\label{Ssum}
\end{eqnarray}
where, with $r:=\exp\left(\tf{1}{2}\xh\right)$, we have
\begin{eqnarray*}
S_{\ell}&=&\sum_{\substack{i,j\ge 1,\\i+j=\ell}}(r^i-r^{-i})(r^j-r^{-j})=
(\ell -1)(r^{\ell}+r^{-\ell})-2\f{r^{\ell-1}-r^{-\ell+1}}{r-r^{-1}}.
\end{eqnarray*}
Now let $\ta:=\tf{1}{2}\xh$, and
substituting this expression for $S_{\ell}$ in~(\ref{Ssum}), we obtain
the revised fourth expression
\begin{eqnarray*}
e^{-Y} x^2 \sum_{i,j\ge 1} j z_{i+j} \f{\pa^2\Phi}{\pa z_i y_j}&=&
2 \sum_{\ell\ge 1} z_{\ell} \xi(\xh)^{\ell-2} u^{\ell-2} t^{\ell}
\bigg( (\ell-1)\ch(\ell\ta)-\f{\sh\big((\ell-1)\ta\big)}{\sh(\ta)}
\bigg).
\end{eqnarray*}
Combining these four expressions, and recalling the
definition~(\ref{E_opsDd}) of the partial differential operator $\De$, we have
\begin{equation}\label{collectTU}
e^{-Y}\De\Phi=\sum_{\ell \ge 1}  z_{\ell}\xi(\xh)^{\ell-3} u^{\ell-2}t^{\ell}T_{\ell}
+\sum_{\ell,m\ge 1} z_{\ell}y_m \xi(\xh)^{\ell+m-3} u^{\ell+m-1}t^{\ell+m}U_{\ell,m},
\end{equation}
where $T_\ell,$ for $\ell\ge1$ and $U_{\ell,m},$ for $\ell,m\ge1,$  are explicit polynomials in
hyperbolic cosines and hyperbolic sines of multiples of $\theta$, and in $\theta,$
using~(\ref{xidef}).   
It is readily shown, using the addition formulae for hyperbolic sines and cosines that 
$T_{\ell}=0$ for $\ell\ge1$ and, similarly, that $U_{\ell,m} = 0$ for $\ell,m\ge 1.$
\dmrjdel{   
\begin{eqnarray*}
T_{\ell}&=&(\ell-1)\ta^{-2}\sh(\ta)\sh(\ell\ta)
+\Big(\ell\ch(\ell\ta)-\ta^{-1}\sh(\ell\ta)\Big)\ta^{-1}\sh(\ta)
+(\ell-2)\ta^{-1}\sh(\ell\ta)\Big(\ch(\ta)-\ta^{-1}\sh(\ta)\Big)\\
&&-\;\ell\ta^{-1}\sh\big((\ell-1)\ta\big)
-2\ta^{-1}\sh(\ta)
\bigg( (\ell-1)\ch(\ell\ta)-\f{\sh\big((\ell-1)\ta\big)}{\sh(\ta)}\bigg)\\
&=&(\ell-2)\ta^{-1}\Big(\ch(\ell\ta)\sh(\ta)+\sh(\ell\ta)\ch(\ta)-\sh\big((\ell-1)\ta\big)
\Big)=0,
\end{eqnarray*}
and, for $\ell,m\ge 1$,
\begin{eqnarray*}
U_{\ell,m}&=&\tf{m+1}{m}\ta^{-3}\sh(\ell\ta)\sh(m\ta)\sh(\ta)
+\ta^{-2}\sh(\ell\ta)\sh(\ta)\Big(\ch(m\ta)-\tf{1}{m}\ta^{-1}\sh(m\ta)\Big)\\
&&+\;\ta^{-2}\sh(\ell\ta)\sh(m\ta)\Big(\ch(\ta)-\ta^{-1}\sh(\ta)\Big)
-\ta^{-2}\sh\big((\ell-1)\ta\big)\sh(m\ta)
-\;\ta^{-2}\sh\big((\ell+m)\ta\big)\sh(\ta)\\
&=&\ta^{-2}\bigg(\sh(\ta)\Big(\sh(\ell\ta)\ch(m\ta)-\sh\big((\ell+m)\ta\big)\Big)
+\sh(m\ta)\Big(\sh(\ell\ta)\ch(\ta)-\sh\big((\ell-1)\ta\big)\Big)
\bigg)=0,
\end{eqnarray*}
where the last equality in each of the above simplifications follows
from~(\ref{shadd}). 
}  
Thus, from~(\ref{collectTU}), we have $\De\Phi=0$.
But $\xi(0)=1$, so $\left.\Phi\right|_{u=0}=z_1t$ and we conclude from
Theorem~\ref{jcut} and the uniqueness of the solution of the
Join-cut Equation that $\Psi=\Phi$, giving the result.
\epf

\subsection{An expression for the coefficients of $\Psi$}
It is now straightforward to determine the coefficients in the generating
series~$\Psi$, and thus obtain a proof of Theorem~\ref{t2}.
\vspace{.05in}

\noindent\underline{\em Proof of Theorem~\ref{t2}}.
Suppose that $\al$ is a partition of $n-i$ with $k$ parts. Then
for all $n\ge i\ge 1$, $k,g\ge 0$, Theorem~\ref{gsallg} and~(\ref{xidef}) gives 
($[A]B$ denotes the coefficient of $A$ in $B$)
\begin{eqnarray*}
[t^nu^{n+k-1+2g}x^{2g}z_iy_{\al}]\Psi
&=&\f{i}{|\Aut \al|}[u^{2g}x^{2g}]\xi(iux)\xi(ux)^{i-2}
\prod_{j=1}^{l(\al)}\xi(\al_j ux)\xi(ux)^{\al_j}\\
&=&\f{i}{|\Aut \al|}[x^{2g}]\xi(x)^{n-2}\xi(ix)
\prod_{j=1}^{l(\al)}\xi(\al_j x)\\
&=&\f{i}{|\Aut \al|}[x^{2g}]\exp\Big(\sum_{j\ge 1}\xi_{2j}q_{2j}(\al\cup i)x^{2j}\Big),
\end{eqnarray*}
so, together with~(\ref{cardiclass}) and~(\ref{gendef}), this gives
\begin{equation*}
c_g(i,\al)
= \f{(n+k-1+2g)!}{n!}\,\al_1\cd\al_k\, i
\sum_{\be\vdash g}\f{\xi_{2\be}q_{2\be}(\al\cup i)}{|\Aut\be|}.
\end{equation*}
But this is symmetric in $\al_1,\ld ,\al_k, i$, and the result
follows immediately by renaming $\al\cup i$ as $\al$, which has $m=k+1$ parts.
\epf
\vspace{.05in}


\section{Further questions}

The following questions arise in the light of the results of this paper:
\begin{enumerate}
\item
Is it possible to find a simple proof of the centrality in Theorem~\ref{t2},
without evaluating the class coefficients $a_g(\al)$? 
This might follow from a decomposition for Young-Jucys-Murphy elements, 
or from a more elementary argument in the symmetric group.

\item
Is it possible to give a direct proof of Corollary~\ref{connS2n-1}
or~\ref{connSn} -- \textit{i.e.}, to establish these relationships between $a_g(\al)$ and $b_g(\al)$
without appealing, as we have, to the explicit formulae? 
This would be particularly interesting, since $b_g(\al)$, as defined, is clearly central. 
Such a proof might involve Young-Jucys-Murphy elements, or a more elementary
argument in the symmetric group, or the geometry of branched covers. Presumably
such a proof would contain a solution to Question 1 above.

\item
In~\cite{gjv}, the polynomiality of $b_g(\al)$ (in the parts of $\al$) in Theorem~\ref{t2sim},
was the basis for a conjectured ELSV-type formula for $H^g_{(n),\al}$, involving a Hodge 
integral over some, unspecified, moduli space.
Does the polynomiality of $a_g(\al)$ in Theorem~\ref{t2} also lead to a similar ELSV-type formula 
when $a_g(\al)$ is rescaled as a covering number?

\end{enumerate}

\section*{Acknowledgements}

The work of both authors was supported by Discovery Grants from NSERC.
We thank John Irving, Igor Pak, Amarpreet Rattan and Ravi Vakil for helpful comments.



\begin{thebibliography}{ELSV}


\bibitem[ELSV]{elsv1} T. Ekedahl, S. Lando, M. Shapiro and A. Vainstein,
\emph{Hurwitz numbers and intersections on moduli spaces of curves},
Invent. \ Math. \ 146 (2001),  297--327.

\bibitem[GJ0]{gj0} I. P. Goulden and D. M. Jackson,
{\em Transitive factorizations into transpositions and holomorphic
mappings on the sphere},
Proc.\  Amer.\  Math.\  Soc.,
{\bf 125} (1997), 51--60.

\bibitem[GJ1]{gj1} I. P. Goulden and D. M. Jackson,
{\em A proof of a conjecture for the number of ramified
coverings of the sphere by the torus},
J.\ Combinatorial Theory (A),
{\bf 88} (1999), 246--258.

\bibitem[GJ3]{gj2} I. P. Goulden and D. M. Jackson,
{\em The number of ramified coverings of the sphere by the
double torus, and a general form for higher genera},
J.\ Combinatorial Theory (A),
{\bf 88} (1999), 259--275.

\bibitem[GJVn]{gjvai}
I. P. Goulden, D. M. Jackson and A. Vainshtein,
{\em The number of ramified coverings of the sphere by the torus and
surfaces of higher genera}, Annals of Combinatorics, {\bf 4} (2000),
27--46. 

\bibitem[GJV]{gjv}
I. P. Goulden, D. M. Jackson, R. Vakil,
{\em Towards the geometry of double Hurwitz numbers},
Adv.\ Math.\ {\bf 198} (2005), 43--92. 

\bibitem[H]{h}
A. Hurwitz,
{\em Ueber Riemann'sche Fl\"{a}chen mit gegebenen Verzweigungspunkten},
Math.\  Ann.\ {\bf 39} (1891), 1--60.

\bibitem[IR]{ir}
J. Irving and A. Rattan,
{\em Factorizations of permutations into star transpositions},\\
math.CO/0610640.

\bibitem[P]{p}
I. Pak,
{\em Reduced decompositions of permutations in terms of star transpositions,
generalized Catalan numbers and $k$-ary trees},
Disc.\ Math.\ {\bf 204} (1999), 329 -- 335.

\bibitem[VO]{vo}
A. M. Vershik and A. Yu. Okounkov,
{\em A new approach to the representation theory of the symmetric groups. II},
math.RT/0503040, v3.

\end{thebibliography}
\end{document}